\documentclass[11pt]{amsart}

\usepackage{amsmath,amssymb,amsthm}
\usepackage{color, tikz}

\theoremstyle{definition}
\newtheorem{theorem}{Theorem}[section]
\newtheorem{prop}[theorem]{Proposition}

\newtheorem{lemma}[theorem]{Lemma}
\newtheorem{definition}[theorem]{Definition}
\newtheorem{corollary}[theorem]{Corollary}
\newtheorem{remark}[theorem]{Remark}
\newtheorem{algorithm}[theorem]{Algorithm}
\newtheorem{observation}[theorem]{Observation}

\newcommand{\Ind}{\mathrm{Ind}}

\newcommand\commentout[1]{}

\begin{document}


\title{Matching and Independence Complexes Related to Small Grids}

\author{Benjamin Braun}
\address{Department of Mathematics\\
         University of Kentucky\\
         Lexington, KY 40506--0027}
\email{benjamin.braun@uky.edu}

\author{Wesley K. Hough}
\address{Department of Mathematics\\
         University of Kentucky\\
         Lexington, KY 40506--0027}
\email{wesley.hough@uky.edu}

\date{3 June 2016}

\thanks{
The first author was partially supported by grant H98230-16-1-0045 from the U.S. National Security Agency.
}

\begin{abstract}
The topology of the matching complex for the $2\times n$ grid graph is mysterious.
We describe a discrete Morse matching for a family of independence complexes $\Ind(\Delta_n^m)$ that include these matching complexes.
Using this matching, we determine the dimensions of the chain spaces for the resulting Morse complexes and derive bounds on the location of non-trivial homology groups for certain $\Ind(\Delta_n^m)$.
Further, we determine the Euler characteristic of $\Ind(\Delta_n^m)$ and prove that several homology groups of $\Ind(\Delta_n^m)$ are non-zero.
\end{abstract}

\maketitle



\section{Introduction}\label{sec:intro}

A \textit{matching} on a simple graph $G$ is a subgraph $H = (V(G),S)$ with $S \subset E(G)$ and maximum vertex degree 1.  
We identify a matching with its edge set $S$.
The \textit{matching complex of $G$}, denoted $M(G)$, is the simplicial complex with vertex set $E(G)$ and faces given by the matchings on $G$.
It is useful to reframe matchings in the language of independent sets as follows.
An \textit{independent set} in a simple graph $G = (V(G),E(G))$ is a set $T\subseteq~V(G)$ such that no two vertices in $T$ are adjacent in $G$.
The \textit{independence complex of $G$}, denoted $\Ind(G)$, is the abstract simplicial complex with vertex set $V(G)$ and faces given by the independent sets in $G$.
Given a simple graph $G$, its \textit{line graph} $L(G)$ has vertex set $V(L(G)) = E(G)$ with two vertices of $L(G)$ adjacent if they are adjacent edges in $G$.  
A key observation is that $M(G) = \Ind(L(G))$ for a finite simple graph $G$.
For the path on $n$ vertices, denoted $Pa_n$, and the cycle on $n$ vertices, denoted $C_n$, the homotopy type of the matching and independence complexes are known \cite[Section~11.4]{JonssonBook}.
However, matching and independence complexes quickly become quite complicated, e.g. \cite{AdamaszekSpecial,BarmakStar,Bouc,BraunIndComplexKneser,WachsDongLaplacian,EarlEtAl,JonssonMoreTorsion,ShareshianWachsTopHomology,VrecicaZivaljevicCycleFree,WachsTopologyMatching}.
Jonsson \cite{JonssonBook} provides a thorough survey regarding these and other simplicial complexes arising from graphs, with special emphasis on the matching complex for complete graphs and complete bipartite graphs.
 
We focus our attention to the grid graphs defined by $V = \{1,2\} \times [n]$ and 
\begin{align*}
E = &  \{ \{(1,k),(2,k)\} : k \in [n]\}  \\
&  \bigcup \, \{ \{(1,k),(1,k+1)\} : k \in [n-1]\}  \\
&  \bigcup \, \{ \{(2,k),(2,k+1)\} : k \in [n-1]\}\, .
\end{align*}
We write $\Gamma_n$ to denote the $2$ by $n+2$ grid graph, e.g. $\Gamma_3$ is isomorphic to:

	
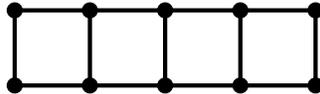
\begin{center}
	\begin{tikzpicture}
	\draw [fill]  (0,0) circle [radius=0.1];
	\draw [fill]  (1,0) circle [radius=0.1]; 
	\draw [fill]  (2,0) circle [radius=0.1];
	\draw [fill]  (3,0) circle [radius=0.1];
	\draw [fill]  (4,0) circle [radius=0.1];
	\draw [fill]  (0,1) circle [radius=0.1];
	\draw [fill]  (1,1) circle [radius=0.1]; 
	\draw [fill]  (2,1) circle [radius=0.1];
	\draw [fill]  (3,1) circle [radius=0.1];
	\draw [fill]  (4,1) circle [radius=0.1];
	\draw [ultra thick] (0,0)--(1,0)--(2,0)--(3,0)--(4,0);
	\draw [ultra thick] (0,1)--(1,1)--(2,1)--(3,1)--(4,1);
	\draw [ultra thick] (0,0)--(0,1);
	\draw [ultra thick] (1,0)--(1,1);
	\draw [ultra thick] (2,0)--(2,1);
	\draw [ultra thick] (3,0)--(3,1);
	\draw [ultra thick] (4,0)--(4,1);	
	\end{tikzpicture}
	\end{center}
We define $D_n := L(\Gamma_n)$, e.g. $D_3$ is isomorphic to 
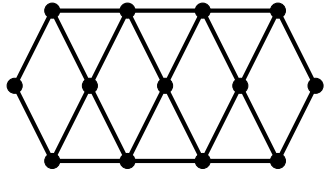
\begin{center}
	\begin{tikzpicture}
	\draw [fill]  (-0.5,1) circle [radius=0.1];
	\draw [fill]  (0.5,1) circle [radius=0.1];
	\draw [fill]  (1.5,1) circle [radius=0.1];
	\draw [fill]  (2.5,1) circle [radius=0.1]; 
	\draw [fill]  (3.5,1) circle [radius=0.1]; 
	\draw [fill]  (0,0) circle [radius=0.1];
	\draw [fill]  (0,2) circle [radius=0.1]; 
	\draw [fill]  (1,0) circle [radius=0.1]; 
	\draw [fill]  (1,2) circle [radius=0.1];
	\draw [fill]  (2,0) circle [radius=0.1]; 
	\draw [fill]  (2,2) circle [radius=0.1]; 
	\draw [fill]  (3,0) circle [radius=0.1];
	\draw [fill]  (3,2) circle [radius=0.1]; 
	\draw [ultra thick] (-0.5,1)--(0,0)--(0.5,1)--(1,0)--(1.5,1)--(2,0)--(2.5,1)--(3,0)--(3.5,1);
	\draw [ultra thick] (-0.5,1)--(0,2)--(0.5,1)--(1,2)--(1.5,1)--(2,2)--(2.5,1)--(3,2)--(3.5,1);
	\draw [ultra thick] (0,0)--(1,0)--(2,0)--(3,0);
	\draw [ultra thick] (0,2)--(1,2)--(2,2)--(3,2);
	\end{tikzpicture}
	\end{center}

In an unpublished manuscript \cite{JonssonMatchingGrids}, Jonsson establishes basic results regarding the matching complexes for $\Gamma_n$ and more general grid graphs.
For example, Jonsson shows that the homotopical depth of $M(\Gamma_n)$ is $\lceil 2n/3\rceil$, which implies that this skeleton of the complex is a wedge of spheres.
However, Jonsson states \cite[page~3]{JonssonMatchingGrids} that ``it is probably very hard to determine the homotopy type of'' matching complexes of grid graphs.

In \cite{LinussonGridGraphs}, Bousquet-M\'{e}lou, Linusson, and Nevo introduce the tool of matching trees for the study of independence complexes.
In this paper, we will use matching trees to produce a Morse matching on the face poset of $M(\Gamma_n)=\Ind(D_n)$.  
Our matching algorithm has a recursive structure that allows us to enumerate the number and dimension of cells in a cellular complex homotopy equivalent to $\Ind(D_n)$.  
We use this recursion to determine topological properties of $\Ind(D_n)$.

Our techniques actually apply to independence complexes of a larger class of graphs that include $D_n$.
Before introducing these graphs, we define two families of related graphs.
First, for $m \geq 1$ and $n \geq 1$, let $Y^m_n$ denote the extended star graph with a central vertex of degree $m$ and paths of with $n$ edges emanating outward.
We refer to one of these paths as a \textit{tendril}. 
(We ignore the degenerate cases $m=0$ and $n=0$.) 
For example, $Y^1_n \cong Pa_{n+1}$,  $Y^2_n \cong Pa_{2n+1}$, and $Y_4^3$ is isomorphic to the following:

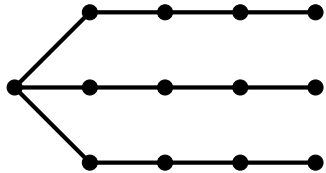
\begin{center}
\begin{tikzpicture}
\draw [fill]  (-1,1) circle [radius=0.1];
\draw [fill]  (0,0) circle [radius=0.1];
\draw [fill]  (1,0) circle [radius=0.1];
\draw [fill]  (2,0) circle [radius=0.1];
\draw [fill]  (3,0) circle [radius=0.1];
\draw [fill]  (0,1) circle [radius=0.1];
\draw [fill]  (1,1) circle [radius=0.1];
\draw [fill]  (2,1) circle [radius=0.1];
\draw [fill]  (3,1) circle [radius=0.1];
\draw [fill]  (0,2) circle [radius=0.1];
\draw [fill]  (1,2) circle [radius=0.1];
\draw [fill]  (2,2) circle [radius=0.1];
\draw [fill]  (3,2) circle [radius=0.1];
\draw [ultra thick] (-1,1)--(0,0)--(1,0)--(2,0)--(3,0);
\draw [ultra thick] (-1,1)--(0,1)--(1,1)--(2,1)--(3,1);
\draw [ultra thick] (-1,1)--(0,2)--(1,2)--(2,2)--(3,2);
\end{tikzpicture}
\end{center}

We further define $\widehat{Y}^m_n$ to be two vertices connected by $m$ disjoint paths each having $n+1$~edges.
(We ignore the degenerate cases $m=0$ and $n=0$.)
For example, $\widehat{Y}^1_n \cong Pa_{n+2}$,  $\widehat{Y}^2_n \cong C_{2n+2}$, and $\widehat{Y}^3_4$ is isomorphic to the following:

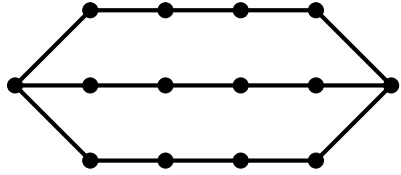
\begin{center}
\begin{tikzpicture}
\draw [fill] (-1,1) circle [radius=0.1];
\draw [fill]  (4,1) circle [radius=0.1];
\draw [fill]  (0,0) circle [radius=0.1];
\draw [fill]  (1,0) circle [radius=0.1];
\draw [fill]  (2,0) circle [radius=0.1];
\draw [fill]  (3,0) circle [radius=0.1];
\draw [fill]  (0,1) circle [radius=0.1];
\draw [fill]  (1,1) circle [radius=0.1];
\draw [fill]  (2,1) circle [radius=0.1];
\draw [fill]  (3,1) circle [radius=0.1];
\draw [fill]  (0,2) circle [radius=0.1];
\draw [fill]  (1,2) circle [radius=0.1];
\draw [fill]  (2,2) circle [radius=0.1];
\draw [fill]  (3,2) circle [radius=0.1];
\draw [ultra thick] (-1,1)--(0,0)--(1,0)--(2,0)--(3,0)--(4,1);
\draw [ultra thick] (-1,1)--(0,1)--(1,1)--(2,1)--(3,1)--(4,1);
\draw [ultra thick] (-1,1)--(0,2)--(1,2)--(2,2)--(3,2)--(4,1);
\end{tikzpicture}
\end{center}

We will impose a specific labeling on this graph throughout this paper: the leftmost vertex is $a$, the rightmost vertex is $b$, and the $k$-th vertex away from $a$ on the $j$-th path is $(j,k)$.
Let $\Delta^m_n$ denote the (labeled) graph $\widehat{Y}^m_{n+1}$ with $n$ additional vertices labeled $\{1, \dots, n\}$ and edges $\{k,(j,k)\}$ and $\{k,(j,k+1)\}$ for each $j \in [m]$ and each $k \in [n]$.  
For example, $\Delta^4_3$ is isomorphic to 

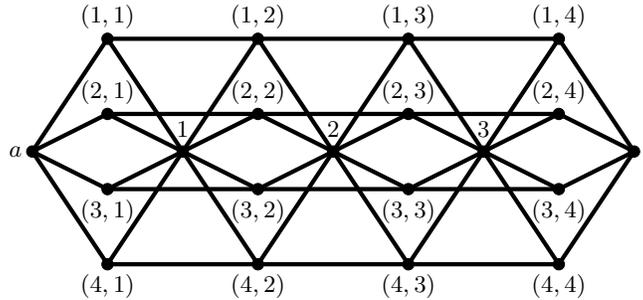
\begin{center}
\begin{tikzpicture}
\draw [fill] (-1,1.5) circle [radius=0.075] node[left] {\footnotesize $a$};
\draw [fill]  (1,1.5) circle [radius=0.075] node[yshift=0.3cm] {\footnotesize $1$};
\draw [fill]  (3,1.5) circle [radius=0.075] node[yshift=0.3cm] {\footnotesize $2$};
\draw [fill]  (5,1.5) circle [radius=0.075] node[yshift=0.3cm] {\footnotesize $3$};
\draw [fill]  (7,1.5) circle [radius=0.075] node[right] {\footnotesize $b$};
\draw [fill]  (0,0) circle [radius=0.075] node[below] {\footnotesize $(4,1)$};
\draw [fill]  (2,0) circle [radius=0.075] node[below] {\footnotesize $(4,2)$};
\draw [fill]  (4,0) circle [radius=0.075] node[below] {\footnotesize $(4,3)$};
\draw [fill]  (6,0) circle [radius=0.075] node[below] {\footnotesize $(4,4)$};
\draw [fill]  (0,1) circle [radius=0.075] node[below] {\footnotesize $(3,1)$};
\draw [fill]  (2,1) circle [radius=0.075] node[below] {\footnotesize $(3,2)$};
\draw [fill]  (4,1) circle [radius=0.075] node[below] {\footnotesize $(3,3)$};
\draw [fill]  (6,1) circle [radius=0.075] node[below] {\footnotesize $(3,4)$};
\draw [fill]  (0,2) circle [radius=0.075] node[above] {\footnotesize $(2,1)$};
\draw [fill]  (2,2) circle [radius=0.075] node[above] {\footnotesize $(2,2)$};
\draw [fill]  (4,2) circle [radius=0.075] node[above] {\footnotesize $(2,3)$};
\draw [fill]  (6,2) circle [radius=0.075] node[above] {\footnotesize $(2,4)$};
\draw [fill]  (0,3) circle [radius=0.075] node[above] {\footnotesize $(1,1)$};
\draw [fill]  (2,3) circle [radius=0.075] node[above] {\footnotesize $(1,2)$};
\draw [fill]  (4,3) circle [radius=0.075] node[above] {\footnotesize $(1,3)$};
\draw [fill]  (6,3) circle [radius=0.075] node[above] {\footnotesize $(1,4)$};
\draw [ultra thick] (-1,1.5)--(0,0)--(2,0)--(4,0)--(6,0)--(7,1.5);
\draw [ultra thick] (-1,1.5)--(0,1)--(2,1)--(4,1)--(6,1)--(7,1.5);
\draw [ultra thick] (-1,1.5)--(0,2)--(2,2)--(4,2)--(6,2)--(7,1.5);
\draw [ultra thick] (-1,1.5)--(0,3)--(2,3)--(4,3)--(6,3)--(7,1.5);
\draw [ultra thick] (0,0)--(1,1.5)--(2,0)--(3,1.5)--(4,0)--(5,1.5)--(6,0);
\draw [ultra thick] (0,1)--(1,1.5)--(2,1)--(3,1.5)--(4,1)--(5,1.5)--(6,1);
\draw [ultra thick] (0,2)--(1,1.5)--(2,2)--(3,1.5)--(4,2)--(5,1.5)--(6,2);
\draw [ultra thick] (0,3)--(1,1.5)--(2,3)--(3,1.5)--(4,3)--(5,1.5)--(6,3);
\end{tikzpicture}
\end{center}

In accordance with this numbering scheme, we define 
\[
\Delta^m_0 := \widehat{Y}^m_1 \text{ and } \Delta^m_{-1} := K_1\, 
\]
where $K_1$ denotes an isolated vertex with no loops.  
It is straightforward to verify that $\Delta^2_n = D_n$, and hence $\Delta_n^m$ is a family generalizing $D_n$.

The article is structured as follows.
In Section~\ref{sec:background} we review discrete Morse theory and matching trees for independence complexes.
In Section~\ref{sec:comb} we describe a matching tree procedure for $\Ind(\Delta_n^m)$ which we call the Comb Algorithm.
This matching tree produces a cellular chain complex $X_n^m$ that is homotopy equivalent to the simplicial chain complex for $\Ind(\Delta_n^m)$.
In Section~\ref{sec:results} we use the Comb Algorithm to establish enumerative properties regarding dimensions of the chain spaces of $X_n^m$.
Finally, in Section~\ref{sec:hom} we apply these enumerative results to derive homological properties of $\Ind(\Delta_n^m)$.
We conclude with two questions for further research.


\section{Discrete Morse Theory}\label{sec:background}

In this section we introduce tools from discrete Morse theory. 
Discrete Morse theory was introduced by R.~Forman in \cite{FormanMorseTheory} and has since become a standard tool in topological combinatorics.  
The main idea of (simplicial) discrete Morse theory is to pair cells in a simplicial complex in a manner that allows them to be cancelled via elementary collapses, reducing the complex under consideration to a homotopy equivalent complex, cellular but possibly non-simplicial, with fewer cells.  
Further details regarding the following definitions and theorems can be found in \cite{JonssonBook} and \cite{KozlovBook}.

\begin{definition}\label{Morsedef}
A \textit{partial matching} in a poset $P$ is a partial matching in the underlying graph of the Hasse diagram of $P$, i.e. it is a subset $\mu \subseteq P \times P$ such that
\begin{itemize}
\item $(a,b)\in \mu$ implies $b$ covers $a$ (sometimes denoted as $a \lessdot b$), i.e. $a<b$ and no $c$ satisfies $a<c<b$, and
\item each $a \in P$ belongs to at most one element in $\mu$.
\end{itemize}
When $(a,b)\in \mu$, we write $a=d(b)$ and $b=u(a)$.  
A partial matching on $P$ is called \textit{acyclic} if there does not exist a cycle 
\[
b_1>d(b_1)<b_2>d(b_2)<\cdots<b_n>d(b_n)<b_1
\] 
with $n\geq 2$ and all $b_i\in P$ being distinct.
\end{definition}

Given an acyclic partial matching $\mu$ on $P$, we say that the unmatched elements of $P$ are \textit{critical}.  
The following theorem asserts that an acyclic partial matching on the face poset of a polyhedral cell complex is exactly the pairing needed to produce our desired homotopy equivalence.

\begin{theorem}\label{mainmorse} {\rm (Main Theorem of Discrete Morse Theory)}
Let $\Delta$ be a polyhedral cell complex, and let $\mu$ be an acyclic partial matching on the face poset of $\Delta$.  Let $c_i$ denote the number of critical $i$-dimensional cells of $\Delta$.  The space $\Delta$ is homotopy equivalent to a cell complex $\Delta_c$ with $c_i$ cells of dimension $i$ for each $i\geq 0$, plus a single $0$-dimensional cell in the case where the empty set is paired in the matching.
\end{theorem}



In \cite{LinussonGridGraphs}, Bousquet-M\'{e}lou, Linusson, and Nevo introduced \emph{matching trees} as a way to apply discrete Morse theory to $\Ind(G)$ for a simple graph $G=(V,E)$.
For $A,B \subseteq V$ such that $A \cap B = \emptyset$, let 
\[
\Sigma(A,B) := \left\{ I \in \mbox{Ind}(G) : A \subseteq I \phantom{.} \mathrm{and} \phantom{.} B\cap I = \emptyset \right\} \, .
\]
For a vertex $p\in V(G)$, let $N(p)$ denote the neighbors of $p$ in $G$.
A \emph{matching tree} $\tau(G)$ for $G$ is a directed tree constructed according to the following algorithm.

\begin{algorithm}[Matching Tree Algorithm, MTA]
Begin by letting $\tau(G)$ be a single node labeled $\Sigma(\emptyset,\emptyset)$, and consider this node a sink until after the first iteration of the following loop:

\textbf{WHILE} $\tau(G)$ has a leaf node $\Sigma(A,B)$ that is a sink with $|\Sigma(A,B)| \geq 2$,

\textbf{DO ONE OF THE FOLLOWING:}

\begin{enumerate}
\item If there exists a vertex $p \in V \setminus (A \cup B)$ such that $|N(p) \setminus (A \cup B)|= 0$, create a directed edge from $\Sigma(A,B)$ to a new node labeled $\emptyset$.  
Refer to $p$ as a \emph{free vertex} of $\tau(G)$.
\medskip
\item[] Since $p \notin A \cup B$, neither $p$ nor any of its neighbors are in $A$.  Moreover, $|N(p)\setminus (A\cup B)|=0$ implies that all neighbors of $p$ are in $B$.  
Consequently, given $\sigma \in \Sigma(A,B)$, we may pair $\sigma$ and $\sigma \cup \{p\}$ in the face poset of $\Ind(G)$.
\medskip
\item If there exist vertices $p\in V\setminus (A\cup B)$ and $v\in N(p)$ such that $N(p)\setminus (A\cup B)=\{v\}$, create a directed edge from $\Sigma(A,B)$ to a new node labeled $\Sigma(A\cup \{v\},B\cup N(v))$.  
Refer to $v$ as a \emph{matching vertex} of $\tau(G)$ with respect to $p$.
\medskip
\item[] Neither $p$ nor any of its neighbors are in $A$, and all of $p$'s neighbors (except for $v$) are in $B$.  
Performing Step 3 (described below) with $v$ implies that the branch with $\Sigma(A,B\cup\{v\})$ has $p$ as a free vertex, so we can perform Step 1 on that branch.
\medskip
\item Choose a vertex $v \in V \setminus (A \cup B)$ and created two directed edges from $\Sigma(A,B)$ to new nodes labeled $\Sigma(A,B \cup \{v\})$ and $\Sigma(A \cup \{v\},B \cup N(v)).$ Refer to $v$ as a \emph{splitting vertex} of $\tau(G)$.
\end{enumerate}

The node $\Sigma(\emptyset,\emptyset)$ is called the \emph{root} of the matching tree, while any non-root node of outdegree $1$ in $\tau(G)$ is called a \emph{matching site} of $\tau(G)$ and any non-root node of outdegree $2$ is called a \emph{splitting site} of $\tau(G)$.
Note that the empty set is always matched at the last node of the form $\Sigma(\emptyset,B)$.
\end{algorithm}

A key observation from \cite{LinussonGridGraphs} is that a matching tree on $G$ yields an acyclic partial matching on the face poset of $\Ind(G)$ as follows.

\begin{theorem}[\cite{LinussonGridGraphs}, Section~2] A matching tree $\tau(G)$ for $G$ yields an acyclic partial matching on the face poset of $\Ind(G)$ whose critical cells are given by the non-empty sets $\Sigma(A,B)$ labeling non-root leaves of $\tau(G)$.  
In particular, for such a set $\Sigma(A,B)$, the set $A$ yields a critical cell in $\Ind(G)$.
\end{theorem}


\section{The Comb Matching Algorithm}\label{sec:comb}

We begin by determining the homotopy type of $\Ind(Y^m_n)$ and $\Ind(\widehat{Y}^m_n)$.  
Since $Y^m_n$ is a tree for $m \geq 1$ and $n \geq 0$, we know by work of Ehrenborg and Hetyei \cite{EhrenborgHetyei} that $\Ind(Y^m_n)$ is either contractible or homotopy equivalent to a single sphere.

\begin{lemma} \label{starlem}
For $m \geq 1$ and $n \geq 0$, 
\[
\Ind(Y^m_n) \simeq
\left\{ \begin{array}{ll}
\ast &\mbox{if } n=3k \\
S^{mk} &\mbox{if } n=3k+1 \\
S^{m(k+1)-1} &\mbox{if } n=3k+2 \\
\end{array} 
\right. \, .
\]
\end{lemma}

\begin{proof}
\emph{Case 1: $n = 3k$.} We use induction on $m$.
If $m = 1$, then $Y^1_n \cong Pa_{3k+1}$; hence, $\Ind(Y^1_n)$ is contractible \cite[Prop~11.16]{KozlovBook}.
Suppose the induction hypothesis holds for $\ell < m$. 
Select a tendril of $Y^m_n$ and label the vertices $1$ through $n$ starting at the leaf. 
We consider a matching tree on $\Ind(Y^m_n)$.  Perform Step~2 of the MTA with $p=1$ and $v=2$.  
Repeat with $p=4$ and $v=5$ and so on modulo $3$.  
Since $n=3k$, we will eventually perform Step~2 with $p = n-2$ and $v = n-1$.  
The remaining subgraph of $Y^m_n$ from which we may select vertices is isomorphic to $Y^{m-1}_n$.  
Since $\Ind(Y^{m-1}_n)$ is contractible by assumption, by induction $\Ind(Y^m_n)$ is contractible as well.

\emph{Case 2: $n=3k+1$ or $n=3k+2$.}
Let $a$ be the vertex of degree $m$ in $Y^m_n$.  
We again consider a matching tree on $\Ind(Y^m_n)$.  
We apply Step~3 of the MTA with $v = a$.  
At the $\Sigma(\{a\}, N(a))$ and $\Sigma(\emptyset, \{a\})$ nodes, the remaining subgraphs of $Y^m_n$ from which we may select vertices are isomorphic to an $m$-fold disjoint union of $Pa_{n-1}$'s and an $m$-fold disjoint union of $Pa_{n}$'s respectively.
When $n=3k+1$, the union of $Pa_n$'s is contractible \cite[Prop~11.16]{KozlovBook}, and each subcomplex $\Ind(Pa_{n-1})$ contributes $\left\lfloor \frac{n-2}{3} \right\rfloor + 1 = k$ vertices toward a single critical cell.  
In total, the vertex $a$ and the vertices from each $\Ind(Pa_{n-1})$ factor combine to form a single critical cell of dimension $mk$.
When $n=3k+2$, the union of the $Pa_{n-1}$'s is contractible \cite[Prop~11.16]{KozlovBook}, and each subcomplex $\Ind(Pa_n)$ contributes $\left\lfloor \frac{n-1}{3} \right\rfloor + 1 = k+1$ vertices toward a single critical cell.  
In total, the vertices from each $\Ind(Pa_n)$ factor combine to form a single critical cell of dimension $m(k+1)-1$.
This gives the result.
\end{proof}

\begin{lemma} \label{paralem}
For $m \geq 2$ and $n \geq 1$, 
\[
\Ind(\widehat{Y}^m_n) \simeq
\left\{ \begin{array}{ll}
S^{mk} &\mbox{if } n=3k \\
S^{mk} &\mbox{if } n=3k+1 \\
S^{mk+1} \vee S^{m(k+1)-1} &\mbox{if } n=3k+2 \\
\end{array} \right. \, .
\]
\end{lemma}

\begin{proof}
In $\widehat{Y}^m_n$, label the two vertices of degree $m$ as $a$ and $b$ respectively.
We consider a matching tree on $\Ind(\widehat{Y}^m_n)$.
First, we apply Step 3 of the MTA with $v=b$.
At the $\Sigma(\{b\}, N(b))$ and $\Sigma(\emptyset, \{b\})$ nodes, the remaining subgraphs of $\widehat{Y}^m_n$ from which we may select vertices are isomorphic to $Y^m_{n-1}$ and $Y^m_n$ respectively.
For $n = 3k$ and $n=3k+1$, the result is immediate from applying Lemma~\ref{starlem} as one of the branches will produce contractible information.

For the $n=3k+2$ case with $m \geq 3$, Lemma~\ref{starlem} only shows that two cells of the appropriate dimension exist, but they may not necessarily form a wedge.
This is sufficient for the remainder of the article, but we prove that the two cells do, in fact, form a wedge for sake of completeness.
Given the matching tree defined above for $\Ind(\widehat{Y}^m_n)$, let $\tau$ denote the cell of dimension $mk+1$, and let $\sigma$ denote the cell of dimension $m(k+1)-1$.
In the style of \cite[Theorem~2.2]{KozlovHopf}, we argue that the feasibility domain of $\sigma$ (see \cite[Def~2.1]{KozlovHopf}) is such that $\tau$ and $\sigma$ must form a wedge.
Suppose there exists a generalized alternating path from $\sigma$ to $\tau$ as per \cite[Def~2.1]{KozlovHopf}.
Our choice of matching tree implies $b \in \tau$ while $b \notin \sigma$.
Let $x_i$ be the last element in the alternating path with $b \notin x_i$, so $b \in x_{i+1}$.
If $x_i \lessdot x_{i+1}$, then $x_i$ and $x_{i+1}$ are matched in the matching tree and so $b$ was designated as a free vertex during some application of Step~1 of the MTA.
This is not possible as $b$ is included in $A \cup B$ in all tree nodes except for the root.
If $x_i > x_{i+1}$, then $x_{i+1} \subseteq x_i$ as sets.
This contradicts that $b \notin x_i$ and $b \in x_{i+1}$.
Consequently, no such generalized alternating path can exist between $\sigma$ and $\tau$.
The feasibility region of $\sigma$ does not contain $\tau$, and so $\sigma$ and $\tau$ form a wedge per \cite[Theorem~2.2]{KozlovHopf}.
\end{proof}

We now develop a matching tree for $\Ind(\Delta^m_n)$.

\begin{algorithm}[Comb Algorithm, CA]
Fix $m \geq 2, n \geq 1$ and use the labeling of the vertices of $\Delta_n^m$ from Section~\ref{sec:intro}.

	\begin{itemize}
	\item[Step 1:] Perform Step 3 of the MTA for $v = 1$, which produces two leaves $\Sigma(\{1\},N(1))$ and $\Sigma(\emptyset,\{1\})$ respectively.
	\item[Step 2:] For each $k \in \{2, \dots, n\}$, inductively perform Step 3 of the MTA for $v = k$ on the leaf $\Sigma(\emptyset,\{1, 2, \dots, k-1\})$, successively producing leaves $\Sigma(\{k\}, N(k) \cup \{1, 2, \dots, k-1\})$ and $\Sigma(\emptyset,\{1, 2, \dots, k\})$.
	\item[Step 3:] At the $\Sigma(\{1\},N(1))$ leaf, we may perform Step 1 of the MTA with $p = a$.
	\item[Step 4:] For each $k \in \{2, \dots, n-1\}$, consider the leaf 
\[
\Sigma(\{k\}, N(k) \cup \{1, 2, \dots, k-1\}) \, .
\]
The remaining subgraph of $\Delta^m_n$ from which we may query vertices is isomorphic to $Y^m_{k-1} \biguplus \Delta^m_{n-(k+1)}$. 
Since $\Ind(Y^m_{k-1})$ is known, we can determine the number and dimension of critical cells below this node by inductively applying this algorithm to $\Delta^m_{n-(k+1)}$.
	\item[Step 5:] At the $\Sigma(\{n\}, N(n) \cup \{1, 2, \dots, n-1\})$ leaf, we may perform Step 1 of the MTA with $p = b$.
	\item[Step 6:] At the $\Sigma(\emptyset,\{1, 2, \dots, n\})$ leaf, the remaining subgraph of $\Delta^m_n$ from which we may query vertices is isomorphic to $\widehat{Y}^m_{n+1}$.  Since $\Ind(\widehat{Y}^m_{n+1})$ is known, we can determine the number and dimension of critical cells arising below this node.
	\end{itemize}
\end{algorithm}

\begin{definition}
Denote by $X^m_n$ the cellular complex arising from the Comb Algorithm applied to $\Ind(\Delta^m_n)$ for $m \geq 2$ and $n \geq 1$.  
As we cannot apply the Comb Algorithm to $\Ind(\Delta^m_0)$, we define $X^m_0$ to be $S^0$, since $\Delta^m_0 \cong \widehat{Y}^m_1$.
\end{definition}

We call this process for generating a matching tree for $\Ind(\Delta_n^m)$ the ``Comb Algorithm'' because of the visual shape of the resulting matching tree.  
Steps~1 and~2 produce the backbone of the ``comb,'' while Steps~3 through~6 produce the teeth.
For example, applying Steps~1 and~2 of the comb algorithm to $\Ind(\Delta_4^m)$ leads to the following (partial) matching tree.

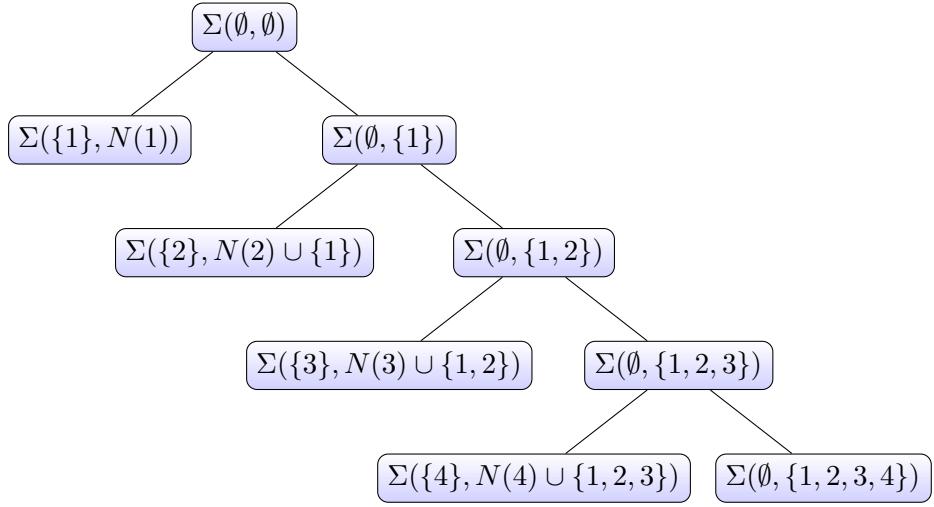
\begin{center}
\begin{tikzpicture}[sibling distance=10em,
  every node/.style = {shape=rectangle, rounded corners,
    draw, align=center,
    top color=white, bottom color=blue!20}]]
  \node {$\Sigma(\emptyset,\emptyset)$}
    child { node {$\Sigma(\{1\},N(1))$} }
    child { node {$\Sigma(\emptyset,\{1\})$}
      child { node {$\Sigma(\{2\},N(2) \cup \{1\})$}}
      child { node {$\Sigma(\emptyset,\{1,2\})$}
        child { node{$\Sigma(\{3\},N(3) \cup \{1,2\})$}}
        child { node{$\Sigma(\emptyset,\{1,2,3\})$}
          child { node{$\Sigma(\{4\},N(4) \cup \{1,2,3\})$}}
          child { node{$\Sigma(\emptyset,\{1,2,3,4\})$}}}}};
\end{tikzpicture}
\end{center}


\section{Chain Spaces of $X_n^m$}\label{sec:results}

Fix $m \geq 2$.
For $d\geq 1$, let $C^d_n$ denote the number of $d$-dimensional cells in $X^m_n$.
Let $C^0_n$ denote one less than the number of $0$-dimensional cells in $X^m_n$.
Since the Comb Algorithm will always pair the empty set with a $0$-cell, we have $C^{-1}_n = 0$.
In this context, $C^d_n = 0$ if $d < 0$ or $n < 0$.

Recall that the \textit{simplicial join} of two abstract simplicial complexes $\Delta$ and $\Gamma$ is the abstract simplicial complex $\Delta \ast \Gamma = \{\sigma \cup \tau | \sigma \in \Delta, \tau \in \Gamma\}$.  
It is clear from this definition that $\Ind(A \biguplus B) \cong \Ind(A) \ast \Ind(B)$ and $M(A \biguplus B) \cong M(A) \ast M(B)$ for graphs $A$ and $B$ and where $\biguplus$ denotes disjoint union.

\begin{prop} \label{cellseed}
Suppose $0 \leq n \leq 3$.  Then $C^d_n = 0$ for all $d \geq 0$ except the following:

\begin{center}
\begin{tabular}{|c|cccccc|} \hline
          & $C^0_0$ & $C^1_1$ & $C^{m-1}_1$ & $C^m_2$ & $C^2_3$ & $C^m_3$ \\ \hline
$m = 2$   &     1   &     2   &         2   &     1   &     2   &     2   \\ \hline
$m \geq 3$ &     1   &     1   &         1   &     1   &     1   &     1   \\ \hline
\end{tabular}
\end{center}
\end{prop}

\begin{proof}
Fix $m \geq 2$.  We separately consider $\Ind(\Delta^m_n)$ for $n \in \{0,1,2,3\}$.

Case 1: Suppose $n=0$.  
Then $\Delta^m_0 \cong \widehat{Y}^m_1$, which implies $\Ind(\Delta^m_0) \simeq S^0$ by Lemma~\ref{paralem}.  
Consequently, $C^0_0 = 1$ while $C^d_0 = 0$ for all other $d$.

Case 2: Suppose $n=1$.  
We apply Step~1 followed by Step~3 of the CA to $\Ind(\Delta^m_1)$.  
At the $\Sigma(\emptyset,\{1\})$ node, the remaining graph from which we may select vertices is isomorphic to $\widehat{Y}^m_2$.  
Thus, $\Ind(\Delta^m_1) \simeq \Ind(\widehat{Y}^m_2)$, from which we can apply Lemma~\ref{paralem}.  
So, $C^1_1 = C^{m-1}_1 = 1$ if $m \geq 3$, and $C^1_1 = 2$ if $m = 2$.  
In either case, $C^d_1 = 0$ for all other $d$.

Case 3: Suppose $n=2$.  
First, apply the Comb Algorithm to $\Ind(\Delta^m_2)$.  
We note that Step~5 subsumes Step~4 in this particular instance.  
Now, Steps~3 and~5 imply that no critical cells are picked out below the nodes $\Sigma(\{1\}, N(1))$ and $\Sigma(\{2\}, N(2) \cup \{1\})$.  
Consequently, Step~6 implies that $\Ind(\Delta^m_2) \simeq \Ind(\widehat{Y}^m_3) \simeq S^m$ via Lemma~\ref{paralem}.  
Thus, $C^m_2 = 1$ while $C^d_2 = 0$ for all other $d$.

Case 4: Suppose $n=3$.  
First, apply the Comb Algorithm to $\Ind(\Delta^m_3)$.  
Now, Steps~3 and~5 imply that no critical cells are picked out below the nodes $\Sigma(\{1\}, N(1))$ and $\Sigma(\{3\}, N(3) \cup \{1,2\})$.  
Per Step~4, at the $\Sigma(\{2\}, N(2) \cup \{1\})$ leaf, the remaining subgraph of $\Delta^m_3$ from which we may query vertices is isomorphic to $Y^m_{1} \biguplus \Delta^m_{0}$.  
We already know that $\Ind(Y^m_{1}) \cong \Ind(\Delta^m_{0}) \simeq S^0$, so $\Ind(Y^m_1 \biguplus \Delta^m_0) \simeq S^2$, so the Comb Algorithm generates a 2-cell below this node.  
At the node $\Sigma(\emptyset, \{1,2,3\})$ generated in Step~6, the remaining subgraph of $\Delta^m_3$ from which we may query vertices is isomorphic to $\widehat{Y}^m_4$.  
Since $\Ind(\widehat{Y}^m_4) \simeq S^m$, the Comb Algorithm generates an $m$-cell below this node.  
In total, we have $C^2_3 = C^m_3 = 1$ if $m \neq 2$, otherwise $C^2_3 = 2$.  
In either case, $C^d_3 = 0$ for all other $d$.
\end{proof}

\begin{theorem}\label{cellrec}
Using Proposition~\ref{cellseed} as initial conditions, for $n \geq 4$ we have
\begin{align}
C^d_n = C^{d-2}_{n-3} + C^{d-(m+1)}_{n-4} + C^{d-m}_{n-3} \, , \label{eqn:rec}
\end{align}
where a summand is zero if the subscript or superscript is negative.
\end{theorem}

\begin{proof}
Assume $n \geq 4$ and $d \geq 0$.  
When applied to $\Ind(\Delta^m_n)$, the Comb Algorithm generates factors of the form $\Ind(Y^m_{k-1} \biguplus \Delta^m_{n-(k+1)})$ homeomorphic to $\Ind(Y^m_{k-1}) \ast \Ind(\Delta^m_{n-(k+1)})$ for $1 \leq k \leq n$.  
We let $C^d_n(k)$ be the number of $d$-dimensional cells in $X^m_n$ produced by the Comb Algorithm below the node $\Sigma(\{k\}, N(k) \cup \{1, 2, \dots, k-1\})$ in $\tau(\Delta^m_n)$, i.e. those referenced in Step~4 of the Comb Algorithm.
We use $C^d_n(\emptyset)$ to denote the number of $d$-dimensional cells arising from Step~6 of the Comb Algorithm.
It is clear that $C^d_n = \displaystyle \sum^n_{k=1} C^d_n(k) + C^d_n(\emptyset)$.

First, whenever $k-1 \equiv 0 \bmod 3$, $\Ind(Y^m_{k-1})$ is contractible and, consequently, so is $\Ind(Y^m_{k-1}) \ast \Ind(\Delta^m_{n-(k+1)})$.  Thus, $C^d_n(k) = 0$ when $k-1 \equiv 0 \bmod 3$, and we may assume that $k = 3 \ell$ or $k = 3 \ell + 2$ for some non-negative integer $\ell$.  Also, note that $\Ind(Y^m_{k-1}) \ast \Ind(\Delta^m_{n-(k+1)})$ is contractible for $k = n$ as $\Ind(\Delta^m_{-1})$ is contractible, i.e. $C^d_n(n) = 0$.  These observations subsume Steps 3 and 5 of the Comb Algorithm.

Next, we consider $C^d_n(2)$. 
 Such a $d$-cell must correspond to the set of $d+1$ vertices consisting of the vertex $2$, a single vertex contributed from $\Ind(Y^m_1)$, and $d-1$ vertices contributed from $\Ind(\Delta^m_{n-3})$.  Therefore, the $d$-cells coming from $\Ind(Y^m_1) \ast \Ind(\Delta^m_{n-3})$ are in bijective correspondence with the $(d-2)$-cells of $\Ind(\Delta^m_{n-3})$, i.e. $C^d_n(2) = C^{d-2}_{n-3}$.  Observe that if $d < 2$, then $C^d_n(2) = 0$.

Similarly, we consider $C^d_n(3)$.  The $d+1$ vertices corresponding to such a $d$-cell consist of the vertex $3$, $m$ vertices contributed from $\Ind(Y^m_2)$ (note that $\Ind(Y^m_2) \simeq S^{m-1}$), and $d-m$ vertices contributed from $\Ind(\Delta^m_{n-4})$, provided $d-m > 0$.  Therefore, the $d$-cells coming from $\Ind(Y^m_2) \ast \Ind(\Delta^m_{n-4})$ are in bijective correspondence with the $(d-(m+1))$-cells of $\Ind(\Delta^m_{n-4})$, i.e. $C^d_n(3) = C^{d-(m+1)}_{n-4}$.  Observe that if $d < m+1$, then $C^d_n(3) = 0$.

Lastly, we simultaneously consider $C^d_n(k)$ for $k \in \{4, 5, \dots, n, \emptyset \}$.  
As before, we can disregard $k \equiv 1 \bmod 3$ and $k = n$.
First, assume that $k = 3\ell$ for some positive integer $\ell$, which implies that $\Ind(Y^m_{k-1}) \simeq S^{m\ell - 1}$.  
Now, consider $C^d_n(k)$.  
A $d$-cell contributed from the factor $\Ind(Y^m_{k-1}) \ast \Ind(\Delta^m_{n-(k+1)})$ consists of the vertex $k$, $m\ell$ vertices from $\Ind(Y^m_{k-1})$, and $d - m \ell$ vertices from $\Ind(\Delta^m_{n-(k+1)})$, provided $d - m\ell > 0$.
We observe that the related factor $\Ind(Y^m_{(k-1)-3} \biguplus \Delta^m_{n-(k+1)})$ is generated when the Comb Algorithm is applied to $\Ind(\Delta^m_{n-3})$.  
It is straightforward to show that the difference in dimension of the critical cell in $\Ind(Y^m_{k-1})$ from that of the critical cell in $\Ind(Y^m_{k-4})$ is $m$.  
This implies that the $d - m\ell$ vertices from $\Ind(\Delta^m_{n-(k+1)})$ that generate a given $d$-cell in the factor $\Ind(Y^m_{k-1}) \ast \Ind(\Delta^m_{n-(k+1)})$ for $\Ind(\Delta^m_n)$ also generate a $(d-m)$-cell in the factor $\Ind(Y^m_{k-4} \biguplus \Delta^m_{n-(k+1)})$ for $\Ind(\Delta^m_{n-3})$ and vice versa.  
Consequently, $C^d_n(k) = C^{d-m}_{n-3}(k-3)$, provided $d \geq m$.  A similar argument holds for $k \equiv 2 \bmod 3$.

Next, we see that $C^d_n(\emptyset)$ = $C^{d-m}_{n-3}(\emptyset)$ if $d \geq m$.  This observation follows because the difference in dimensions of the critical cells in $\Ind(\widehat{Y}^m_{n+1})$ from those of the critical cells in $\Ind(\widehat{Y}^m_{n-2})$ is $m$ while the number of critical cells is constant modulo 3.

Therefore, $\displaystyle \sum^n_{k=4} C^d_n(k) + C^d_n(\emptyset) = \displaystyle \sum^{n-3}_{k=1} C^{d-m}_{n-3}(k) + C^{d-m}_{n-3}(\emptyset) = C^{d-m}_{n-3}$, and we conclude that 
\[
C^d_n = \sum^n_{k=1} C^d_n(k) + C^d_n(\emptyset) = C^{d-2}_{n-3} + C^{d-(m+1)}_{n-4} + C^{d-m}_{n-3} \, .
\]
\end{proof}

Let $\chi^m_n$ denote the reduced Euler characteristic of $X^m_n$.
Note that since $C_n^0$ is one less than the number of zero-dimensional cells in $X_n^m$, we have $\chi^m_n = \sum_{d \geq 0} (-1)^d C^d_n$.

\begin{corollary}
Given the initial conditions from Proposition~\ref{cellseed}, if $m \geq 2$ and $n \geq 4$, then $$\chi^m_n = (1 + (-1)^m) \chi^m_{n-3} + (-1)^{m+1}\chi_{n-4}^m.$$ 
\end{corollary}

\begin{proof}
Fix $m$ and $n$ as above.
Using the recursion \eqref{eqn:rec} for $C_n^d$, we obtain
\begin{align*}
\chi^m_n =& \sum_{d \geq 0} (-1)^d C^d_n \\
=& \sum_{d \geq 0} (-1)^d \left( C^{d-2}_{n-3} + C^{d-(m+1)}_{n-4} +C^{d-m}_{n-3} \right) \\
=&  \left( \sum_{d \geq 0} (-1)^d C^{d-2}_{n-3} \right) + \left( \sum_{d \geq 0} (-1)^d C^{d-(m+1)}_{n-4} \right) \\
& \hspace{1cm}+ \left( \sum_{d \geq 0} (-1)^d C^{d-m}_{n-3} \right) \\
=&  \left( \sum_{d \geq 0} (-1)^{d-2} C^{d-2}_{n-3} \right) + \left( (-1)^{m+1} \sum_{d \geq 0} (-1)^{d-(m+1)} C^{d-(m+1)}_{n-4} \right) \\
 & \hspace{1cm} + \left( (-1)^m \sum_{d \geq 0}(-1)^{d-m} C^{d-m}_{n-3} \right) \\
=&  \left( \sum_{d \geq 0} (-1)^d C^d_{n-3} \right) + \left( (-1)^{m+1} \sum_{d \geq 0} (-1)^d C^d_{n-4} \right) \\
& \hspace{1cm}+ \left( (-1)^m \sum_{d \geq 0}(-1)^d C^d_{n-3} \right) \\
=&  \chi^m_{n-3} + (-1)^{m+1} \chi^m_{n-4} + (-1)^m \chi^m_{n-3} \\
=&  (1 + (-1)^m)\chi^m_{n-3} + (-1)^{m+1} \chi^m_{n-4}
\end{align*}
The fifth equality in the above list is obtained via reindexing and the observations that $C^{d-(m+1)}_{n-4}$ = 0 for $d < m$ and $C^{d-m}_{n-3} = 0$ for $d < m-1$.
\end{proof}

\begin{corollary}
When $m$ is even, $\chi^m_n$ satisfies the recursion $a_n = a_{n-3} - a_{n-2} - a_{n-1}$ with initial conditions $a_0 = 1$, $a_1 = -2$, and $a_2 = 1$, and hence has generating function $\frac{1-x}{1+x+x^2-x^3}$.
\end{corollary}

\begin{proof} 
Assume that $m \geq 2$ is even.  
First, observe that $\chi^m_0$ = 1, $\chi^m_1 = -2$, and $\chi^m_2 = 1$ by Proposition~\ref{cellseed}, so both relations have the same initial conditions.  
We can easily verify that $\chi^m_3 = 2 = 1 - (-2) -1 = a_0  - a_1 - a_2 = a_3$.  
Now, for fixed $n$, assume that $\chi^m_{\ell}$ satisfies both relations for $\ell < n$. 
Since $m$ is even, we have that $\chi^m_n = 2 \cdot \chi^m_{n-3} - \chi^m_{n-4} = \chi^m_{n-3} + (\chi^m_{n-3} - \chi^m_{n-4})$.  
By assumption, $\chi^m_{n-1} = \chi^m_{n-4} - \chi^m_{n-3} - \chi^m_{n-2}$, which implies that $\chi^m_{n-3} - \chi^m_{n-4} = - \chi^m_{n-2} - \chi^m_{n-1}.$  
Therefore, $\chi^m_n = \chi^m_{n-3} + (\chi^m_{n-3} - \chi^m_{n-4}) = \chi^m_{n-3} - \chi^m_{n-2} - \chi^m_{n-1}$, i.e. $\chi^m_n$ satisfies both relations by induction.
(This sequence is the A078046 entry in the OEIS \cite{OEIS}.)
\end{proof}

\begin{remark}
When $m$ is odd, $\chi^m_n = \chi^m_{n-4}$.  
It is easy to verify that $\chi^m_0 = 1$, $\chi^m_1 = 0$, $\chi^m_2 = -1$, and $\chi^m_3 = 1$ from Proposition~\ref{cellseed}.  
Therefore, $\chi^m_n \in \{-1,0,1\}$ depending on the value of $n$ modulo $4$. 
\end{remark}

When $m=2$, the dimensions of $C_n^d$ have an interesting enumerative interpretation.
The sequence A201780 in OEIS \cite{OEIS} is the Riordan array of 
\[
\left( \dfrac{(1-x)^2}{1-2x}, \dfrac{x}{1-2x} \right)
\]
 which can be alternatively defined by
\begin{equation} \label{riordan}
T(j,k) = 2 \cdot T(j-1,k) + T(j-1,k-1)
\end{equation}
with initial conditions $T(0,0) = 1$, $T(1,0) = 0$, $T(2,0) = 1$, and $T(j,k) = 0$ if $k<0$ or $j<k$.

\begin{prop}
When $m=2$, \eqref{eqn:rec} reduces to $C_n^d = 2 C^{d-2}_{n-3} + C^{d-3}_{n-4}$.  Consequently, $C^d_n = T(n-d+2,3d-2n)$, while $T(j,k) = C_{2(j-2)+k}^{3(j-2)+k}$.
\end{prop}

\begin{proof}

The initial conditions of $C_d^n$ are realized as entries in this Riordan array as follows.
We have $C_0^0 = 1 = T(2,0)$, and
\begin{align*}
C_1^1 &= 2\\
 &= 2(2 \cdot 0 + 1) + 0 \\
&= 2(2 \cdot T(0,1) + T(0,0))+T(1,0) \\
&= 2 \cdot T(1,1) + T(1,0) \\
&= T(2,1)
\end{align*}

\begin{align*}
C_2^2 &= 1 \\
&= 2 \cdot 0 + 1 \\
&= 2 \cdot T(1,2) + T(1,1)\\
&= T(2,2)
\end{align*}

\begin{align*}
C_2^3 &= 2 \\
& = 2 \cdot 1 + 0 \\
&= 2 \cdot T(2,0) + T(2,-1) \\
&= T(3,0)
\end{align*}

Define functions $J(x,y)=y-x+2$ and $K(x,y)=3x-2y$.
Observe that $J(d,n) = n-d+2$, $J(d-2,n-3) = n-d+1$, and $J(d-3,n-4) = n-d+1$.
Similarly, $K(d,n) = 3d-2n$, $K(d-2,n-3) = 3d-2n$, and $K(d-3,n-4) = 3d-2n-1.$
Applying the relation \eqref{riordan} to the entry $T(J(d,n),K(d,n)) = T(n-d+2,3d-2n)$ gives
\begin{align*}
& T(n-d+2,3d-2n) = \\
& 2 \cdot T(n-d+1,3d-2n) + T(n-d+1,3d-2n-1) =\\
& 2 \cdot T(J(d-2,n-3),K(d-2,n-3)) + T(J(d-3,n-4),K(d-3,n-4))
\end{align*}
Thus, the recursion applied to $T(n-d+2,3d-2n)$ matches that of $C_d^n$.

The proof of the second half of the claim is similar and omitted.
\end{proof}


\section{Homological Properties of $X_n^m$}\label{sec:hom}

In this section we consider homological implications of the comb algorithm.

\begin{theorem} \label{mindim}
Fix $m \geq 2$ and $n \geq 0$.  Define 
\[
d^{min}_n := \left\{
\begin{array}{ll}
\left\lfloor \dfrac{2n+2}{3} \right\rfloor &\mbox{if } n=3k \mbox { or } n=3k+1 \\
2 \left\lfloor \dfrac{n-1}{3} \right\rfloor + m & \mbox{if } n=3k+2 \\
\end{array} \right. \, .
\]

Then, $C^d_n = 0$ if $0 \leq d < d^{min}_n$, excluding the base $0$-cell.  When $m=2$, these two formulas coincide.
\end{theorem}

\begin{proof}
By Proposition~\ref{cellseed}, the claim holds for $n \in \{0,1,2,3\}$.  
We proceed by induction.

For $n \geq 4$, suppose that the claim is true for all $0 \leq i < n$.  
Consider the leaf $\Sigma(\{j\}, N(j) \cup \{1, 2, \dots, j-1\})$ from the Comb Algorithm applied to $\Ind(\Delta^m_n)$.  
Steps~3 and~4 of the Comb Algorithm allow us to assume that $j \in \{2, \dots, n\}$.  
If $j < n$, then the remaining subgraph of $\Delta^m_n$ from which we may query vertices is isomorphic to $Y^m_{j-1} \biguplus \Delta^m_{n-(j+1)}$, which corresponds to a subcomplex of $\Ind(\Delta^m_n)$ of the form $\Ind(Y^m_{j-1})~\ast~\Ind(\Delta^m_{n-(j+1)})$.
Moreover, by Lemma~\ref{starlem}, $\Ind(Y^m_{j-1})$ is contractible when $j \bmod 3 \equiv 1$.  
Since joins respect homotopy equivalences, $\Ind(Y^m_{j-1}) \ast \Ind(\Delta^m_{n-(j+1)})$ is contractible when $j \bmod 3 \equiv 1$, thus we may assume that $j = 3\ell$ or $j = 3\ell + 2$ for some non-negative integer $\ell$.
Observe that when $j=3\ell$ or $3\ell+2$, $\Ind(Y^m_{j-1})$ is homotopy equivalent to $S^{m\ell-1}$ or $S^{m\ell}$, respectively. 
We will let $\delta_j$ denote the dimension of this sphere.

Now, consider $j \in \{2, \dots, n-1\}$. 
Since $n-(j+1) < n$, the induction hypothesis holds for $\Ind(\Delta^m_{n-(j+1)})$.
We count the minimum number of vertices in a critical cell in the matching tree below the node $\Sigma(\{j\}, N(j) \cup \{1, 2, \dots, n-1\})$, namely:
	\begin{itemize}
	\item 1 vertex for $j$ itself
	\item $\delta_j + 1$ vertices from $\Ind(Y^m_{j-1})$
	\item $d^{min}_n + 1$ vertices from $X^m_{n-(j+1)}$
	\end{itemize}
The total number of vertices corresponds to a cell of dimension $\delta_j + d^{min}_n + 2$ below the node $\Sigma(\{j\}, N(j) \cup \{1, 2, \dots, j-1\})$.

As an aside, if $j = n$, the remaining subgraph of $\Delta^m_n$ from which we may query vertices is isomorphic to $\widehat{Y}^m_{n+1}$, so we can also expect the subcomplex $\Ind(\widehat{Y}^m_{n+1})$ to contribute one or two cells of appropriate dimension per Lemm~\ref{paralem}. 

Now, we consider three cases.




\emph{Case: $n = 3k$.}
The proposed $d^{min}_n$ is $\left\lfloor \frac{2n+2}{3} \right\rfloor = \left\lfloor \frac{6k+2}{3} \right\rfloor = 2k$.
If $j = 3\ell$, then we have $n-(j+1)=3(k-\ell-1)+2$, hence $d^{min}_{n-(j+1)} = 2(k-\ell-1)+m$.
Thus, $\delta_j + d^{min}_{n-(j+1)} + 2 = (m\ell-1) + 2(k-\ell-1)+m + 2 = 2k + (m-2)\ell + (m-1)$.
If $j = 3\ell+2$, then we have $n-(j+1)=3(k-\ell-1)$, hence $d^{min}_{n-(j+1)} = 3(k-\ell-1)$.
Thus, $\delta_j + d^{min}_{n-(j+1)} + 2 = (m\ell) + 2(k-\ell-1) + 2 = 2k + (m-2)\ell$.

The cell contributed by the subcomplex $\Ind(\widehat{Y}^m_{n+1})$ is of dimension $mk$.

Observe that these three values are greater than or equal to $2k$ as $m \geq 2$.
Therefore, none of the cells in $X^m_n$ are of dimension smaller than $\left\lfloor \frac{2n+2}{3} \right\rfloor$.  
Further, when $j=2$, we have that the factor $\Ind(Y^m_1) \ast \Ind(\Delta^m_{n-3})$ produces a cell of dimension exactly $2k$.

\emph{Case: $n = 3k+1$.}  
The proposed $d^{min}_n$ is $\left\lfloor \frac{2n+2}{3} \right\rfloor = \left\lfloor \frac{6k+4}{3} \right\rfloor = 2k+1$.
If $j = 3\ell$, then we have $n-(j+1))=3(k-\ell)$, hence $d^{min}_{n-(j+1)} = 2(k-\ell)$.
Thus, $\delta_j + d^{min}_{n-(j+1)} + 2 = (m\ell-1) + 2(k-\ell) + 2 = 2k + (m-2)\ell + 1$.
If $j = 3\ell+2$, then we have  $n-(j+1))=3(k-\ell-1) + 1$, hence $d^{min}_{n-(j+1)} = 2(k-\ell)-1$.
Thus, $\delta_j + d^{min}_{n-(j+1)} + 2 = (m\ell) + 2(k-\ell)-1 + 2 = 2k + (m-2)\ell + 1$.

The cells contributed by the subcomplex $\Ind(\widehat{Y}^m_{n+1})$ are of dimensions $mk+1$ and $m(k+1)-1$.

Observe that these four values are greater than or equal to $2k+1$ as $m \geq 2$.  
Therefore, none of the cells in $X^m_n$ are of dimension smaller than $\left\lfloor \frac{2n+2}{3} \right\rfloor$.  
Further, when $j=2$, we have that the factor $\Ind(Y^m_1) \ast \Ind(\Delta^m_{n-3})$ produces a cell of dimension exactly $2k+1$.

\emph{Case: $n = 3i+2$.} 
The proposed $d^{min}_n$ is $2 \left\lfloor \frac{n-1}{3} \right\rfloor + m = 2\left\lfloor \frac{3k+1}{3} \right\rfloor + m = 2k+m$.
If $j = 3\ell$, then we have $n-(j+1)=3(k-\ell)+1$, hence $d^{min}_{n-(j+1)} =2(k-\ell)+1$.
Thus, $\delta_j + d^{min}_{n-(j+1)} + 2 = (m\ell-1) + 2(k-\ell)+1 + 2 = 2k + (m-2)\ell + 2$.  
Because $j = 3\ell$ and $j \geq 2$, we have $\ell \geq 1$, which implies that  $2k + (m-2)\ell + 2 \geq 2k + (m-2) + 2 = 2k + m$.
If $j = 3\ell+2$, then we have $n-(j+1))=3(k-\ell-1)+2$, hence $d^{min}_{n-(j+1)} =2(k-\ell-1)+m$.
Thus, $\delta_j + d^{min}_{n-(j+1)} + 2 = (m\ell) + 2(k-\ell-1)+m + 2 = 2k + (m-2)\ell + m$.

The cell contributed by the subcomplex $\Ind(\widehat{Y}^m_{n+1})$ is of dimension $m(k+1)$.

Observe that these three values are greater than or equal to $2k+m$ as $m \geq 2$.  
Therefore, none of the cells in $X^m_n$ are of dimension smaller than $2 \left\lfloor \frac{n-1}{3} \right\rfloor + m$.  
Further, when $j=2$, we have that the factor $\Ind(Y^m_1) \ast \Ind(\Delta^m_{n-3})$ produces a cell of dimension exactly $2k+m$.
\end{proof}

\begin{remark}
Theorem~\ref{mindim} shows that $X^m_n$ is at least $d^{min}_n$-connected. 
After a suitable adjustment of notation, this agrees with results of Jonsson \cite[Proposition~2.7]{JonssonMatchingGrids} regarding the connectivity of $\Ind(\Delta_n^2)$.
\end{remark}

\begin{theorem} \label{maxdim}
Fix $m \geq 2$ and $n \geq 0$.  Define 
\[
d^{max}_n := \left\{
\begin{array}{ll}
\left\lfloor \dfrac{3n+2}{4} \right\rfloor &\mbox{if } m = 2 \\
n + 1 + (m-3)\left\lfloor \dfrac{n+2}{3} \right\rfloor&\mbox{otherwise} \\
\end{array} \right. \, .
\]

Then, $C^d_n = 0$ if $d > d^{max}_n$.
\end{theorem}

\begin{proof} 
By Proposition~\ref{cellseed}, the claim holds for $n \in \{0,1,2,3\}$.
We proceed by induction.

Assume $n \geq 4$, and suppose the claim is true for all $0 \leq i < n$.
Consider the maximum dimension of a cell produced below the node $\Sigma(\{j\}, N(j) \cup \{1, 2, \dots, j-1\})$ from the Comb Algorithm applied to $\Ind(\Delta^m_n)$.  As before, we may assume $j \in \{2, \dots, n\}$.  If $j=n$, the remaining subgraph of $\Delta^m_n$ from which we may query vertices is isomorphic to $\widehat{Y}^m_{n+1}$.  If $j < n$, then the remaining subgraph is $Y^m_{j-1} \biguplus \Delta^m_{n-(j+1)}$, which corresponds to a subcomplex of $\Ind(\Delta^m_n)$ of the form $\Ind(Y^m_{j-1}) \ast \Ind(\Delta^m_{n-(j+1)})$.
We will again use the notation $\delta_j$ from the proof of Theorem~\ref{mindim}.

Consider $j \in \{2, \dots, n-1\}$, which implies $n-(j+1) < n$ so that the induction hypothesis holds for $\Ind(\Delta^m_{n-(j+1)})$.
We count the maximum number of vertices in a critical cell in the matching tree below the node $\Sigma(\{j\}, N(j) \cup \{1, 2, \dots, n-1\})$.  Namely,
	\begin{itemize}
	\item 1 vertex for $j$ itself
	\item $\delta_j + 1$ vertices from $\Ind(Y^m_{j-1})$
	\item $d^{max}_{n-(j+1)} + 1$ vertices from $X^m_{n-(j+1)}$
	\end{itemize}
The total number of vertices corresponds to a cell of dimension $\delta_j + d^{max}_{n-(j+1)} + 2$ below the node $\Sigma(\{j\}, N(j) \cup \{1, 2, \dots, j-1\})$.

\emph{Case: $m = 2$.}

The proposed $d^{max}_n$ is $\left\lfloor \frac{3n+2}{4} \right\rfloor$.

If $j=3\ell$, then $\delta_j + d^{max}_{n-(j+1)} + 2 = (2\ell-1) + \left\lfloor \frac{3(n-(3\ell+1))+2}{4} \right\rfloor + 2 = \left\lfloor \frac{3n-\ell+3}{4} \right\rfloor \leq d^{max}_n$ as $\ell \geq 1$ because $j \geq 2$.
If $j=3\ell+2$, then $\delta_j + d^{max}_{n-(j+1)} + 2 = 2\ell + \left\lfloor \frac{3(n-(3\ell+3)+2}{4} \right\rfloor + 2 = \left\lfloor \frac{3n-\ell+1}{4} \right\rfloor \leq d^{max}_n$.

We now consider the contribution of the subcomplex corresponding to $\Ind(\widehat{Y}^m_{n+1})$.
When $n=3k$, $d^{max}_{n} = \left\lfloor \frac{9k+2}{4} \right\rfloor \geq 2k$ while the $\widehat{Y}^m_{n+1}$ contribution has dimension $2k$.
When $n=3k+1$, $d^{max}_{n} = \left\lfloor \frac{9k+5}{4} \right\rfloor \geq 2k+1$ while the $\widehat{Y}^m_{n+1}$ contributions have dimension $2k+1$.
When $n=3k+2$, $d^{max}_{n} = \left\lfloor \frac{9k+8}{4} \right\rfloor \geq 2k+2$ while the $\widehat{Y}^m_{n+1}$ contribution has dimension $2k+2$.

All considered, no cells of $X^m_n$ exceed the proposed maximum dimension.

\emph{Case: $m \geq 3$}.
The proposed $d^{max}_n$ is $n + 1 + (m-3)\left\lfloor \frac{n+2}{3} \right\rfloor$.
If $j=3\ell$, then $\delta_j~+~d^{max}_{n-(j+1)}~+~2~= (m\ell-1) + \left( n-(3\ell+1) + 1 + (m-3) \left\lfloor \frac{n-(3\ell+1)+2}{3} \right\rfloor \right) + 2 = n + 1 + (m-3) \left( \left\lfloor \frac{n-3\ell+1}{3} \right\rfloor + \ell \right) = n + 1 + (m-3) \left\lfloor \frac{n+1}{3} \right\rfloor \leq d^{max}_n$.
If $j=3\ell+2$, then $\delta_j + d^{max}_{n-(j+1)} + 2 = (m\ell) + \left( n-(3\ell+3) + 1 + (m-3) \left\lfloor \frac{n-(3\ell+3)+2}{3} \right\rfloor \right) + 2 = n + (m-3) \left( \left\lfloor \frac{n-3\ell-1}{3} \right\rfloor + \ell \right) = n + (m-3) \left\lfloor \frac{n-1}{3} \right\rfloor \leq d^{max}_n$.

We now consider the contribution of the subcomplex corresponding to $\Ind(\widehat{Y}^m_{n+1})$.
When $n=3k$, $d^{max}_{n} = 3k + 1 + (m-3)\left\lfloor \frac{3k+2}{3} \right\rfloor = mk+1$ while the $\widehat{Y}^m_{n+1}$ contribution has dimension $mk$.
When $n=3k+1$, $d^{max}_{n} = (3k+1) + 1 + (m-3)\left\lfloor \frac{(3k+1)+2}{3} \right\rfloor= m(k+1)-1$ while the $\widehat{Y}^m_{n+1}$ contributions have dimension $mk+1$ and $m(k+1)-1$ respectively.
When $n=3k+2$, $d^{max}_{n} = (3k+2) + 1 + (m-3)\left\lfloor \frac{(3k+2)+2}{3} \right\rfloor = mk+m$ while the $\widehat{Y}^m_{n+1}$ contribution has dimension $mk+m$.

All considered, no cells of $X^m_n$ exceed the proposed maximum dimension.

Observe that, in both cases, if $j=3$, then the $\Ind(Y^m_2) \ast \Ind(\Delta^m_{n-4})$ factor produces a cell of dimension exactly $d^{max}_n$.
\end{proof}

Using Theorems~\ref{cellseed} and~\ref{cellrec} we can create data tables containing dimensions of the integral cellular chain spaces of $X^m_n$ for reasonable values of $n$ and $m$.  
For $m \geq 4$, it is interesting that gaps appear in the dimensions of the chain spaces for low values of $d$ relative to $n$.  
For example, the Comb Algorithm eliminates all cells of dimension $\left\lfloor \frac{2n+2}{3} \right\rfloor + 1, \dots, \left\lfloor \frac{2n+2}{3} \right\rfloor + (m-3)$ when $n=3k$ or $n=3k+1$.
Further, we can explicitly determine the lowest non-vanishing homology for $n=3k$ and $n=3k+1$ when $m \geq 4$; see Jonsson \cite[Lemma~2.3 and Proposition~2.7]{JonssonMatchingGrids} for analogous results when $m=2$.

\begin{theorem}
Suppose that $m \geq 4$, and let $d_n = \left\lfloor \frac{2n+2}{3} \right\rfloor$.  If $n = 3k$ or $n=3k+1$, then $H_{d_n}(X^m_n; \mathbb{Z}) \cong \mathbb{Z}$.  If $n = 3k+2$, then $H_{d_n}(X^m_n; \mathbb{Z})$ is trivial.
\end{theorem}

\begin{proof}
Recall from Theorem~\ref{mindim} that for $n=3k+2$ and $m \geq 3$, the minimum dimension of critical cells produced by the Comb Algorithm is $2\left\lfloor \frac{n-1}{3} \right\rfloor + m$.
It is easy to verify that $\left\lfloor \frac{2n+2}{3} \right\rfloor = 2k+2 < 2k+m =  2\left\lfloor \frac{n-1}{3} \right\rfloor + m$.  
Therefore, $C^{d_n}_n = 0$ when $n=3k+2$, i.e. $H_{d_n}(X^m_n; \mathbb{Z})$ is trivial.

Now, assume that $n = 3k$.  
We know that $C^{\ell}_n = 0$ for $\ell < d_n$ from our cellular dimension range.  
We argue by induction on $k$ that $C^{d_n}_n = 1$ while $C^{d_n+1}_n = 0$, which proves the claim for $n=3k$.  
Begin by recalling that $C^0_0 = 1$ and $C^1_0 = 0$, which provides a base case.  
Assume that $C^{d_{3\ell}}_{3\ell} = 1$ while $C^{d_{3\ell}+1}_{3\ell} = 0$ for $0 \leq \ell < k$. 
We know that $C^{d_n}_n = C^{d_{3k}}_{3k} = C^{d_{3k}-2}_{3k-3} + C^{d_{3k}-m-1}_{3k-4} + C^{d_{3k}-m}_{3k-3}$ by our cellular recursion.  
Observe that $d_{3k}-2 = 2k-2 = d_{3k-3}$, so $C^{d_{3k}-2}_{3k-3} = 1$ by the induction hypothesis.  
Since $d_{3k}-m-1 < d_{3k}-2 = d_{3k-4}$, it follows that $C^{d_{3k}-m-1}_{3k-4} = 0$.  
Similarly, $d_{3k}-m < d_{3k}-2 = d_{3k-3}$, so $C^{d_{3k}-m}_{3k-3} = 0$.  
Hence, $C^{d_n}_n = 1$.

We also know that $C^{d_n+1}_n = C^{d_{3k}+1}_{3k} = C^{d_{3k}-1}_{3k-3} + C^{d_{3k}-m}_{3k-4} + C^{d_{3k}-m+1}_{3k-3}$ by our cellular recursion.  
Observe that $d_{3k}-1 = 2k-1 = d_{3k-3} + 1$, so $C^{d_{3k}-1}_{3k-3} = 0$ by the induction hypothesis.  
Now, $d_{3k}-m < d_{3k}-2 = d_{3k-4}$ still, which implies $C^{d_{3k}-m}_{3k-4} = 0$.  
Similarly, $d_{3k}-m + 1 < d_{3k}-2 = d_{3k-3}$, so $C^{d_{3k}-m+1}_{3k-3} = 0$.  
Hence, $C^{d_n+1}_n = 0$.

By induction, we conclude that $C^{d_{3k}}_{3k} = 1$ while $C^{d_{3k}+1}_{3k} = 0$ for all $k$, from which the result follows.

Assume that $n = 3k+1$; this argument is similar to the previous case.
We again argue by induction on $k$ that $C^{d_n}_n = 1$ while $C^{d_n+1}_n = 0$.  
We obtain our base case by recalling that $C^1_1 = 1$ and $C^2_1 = 0$ for $m \geq 4$.  
Next, we know that $C^{d_n}_n = C^{d_{3k+1}}_{3k+1} = C^{d_{3k+1}-2}_{3k-2} + C^{d_{3k+1}-m-1}_{3k-3} + C^{d_{3k}-m}_{3k-3}$ by our cellular recursion.  
Observe that $d_{3k+1}-2 = 2k-1 = d_{3k-2} = d_{3(k-1)+1}$, so $C^{d_{3k+1}-2}_{3k-2} = 1$ by the induction hypothesis.  
Now, $d_{3k+1}-m-1 = 2k-m  < 2k-2 = d_{3k-3}$, which implies $C^{d_{3k+1}-m-1}_{3k-3} = 0$.  
Similarly, $d_{3k+1}-m = 2k-m + 1 < 2k-1 = d_{3k-2}$, so $C^{d_{3k+1}-m}_{3k-2} = 0$.  
Hence, $C^{d_n}_n = 1$.

We also know that $C^{d_n+1}_n = C^{d_{3k+1}+1}_{3k+1} = C^{d_{3k+1}-1}_{3k-2} + C^{d_{3k+1}-m}_{3k-3} + C^{d_{3k+1}-m+1}_{3k-2}$ by our cellular recursion.  
Observe that $d_{3k+1}-1 = d_{3k+1} - 2 + 1 = d_{3k-2} + 1$, so $C^{d_{3k+1}-1}_{3k-2} = 0$ by the induction hypothesis.  
Now, $d_{3k+1}-m = 2k - m + 1 <  2k-2 = d_{3k-3}$ still, which implies $C^{d_{3k+1}-m}_{3k-3} = 0$. 
Similarly, $d_{3k+1}-m + 1 < 2k - 1 = d_{3k-2}$, so $C^{d_{3k+1}-m+1}_{3k-2} = 0$. 
Hence, $C^{d_n+1}_n = 0$.

By induction, we conclude that $C^{d_{3k+1}}_{3k+1} = 1$ while $C^{d_{3k+1}+1}_{3k+1} = 0$ for all $k$, from which the result follows.
\end{proof}

For other homology groups, the Comb Algorithm provides less comprehensive results.
For example, when $m=2$, that is, when $X^m_n$ is homotopy equivalent to the matching complex on the $2 \times (n+2)$ grid graph, a direct analysis of the chain space dimensions on a data table yields the following.

\begin{observation}
$X_n^2$ has non-trivial free integral homology in dimension $\left\lfloor \frac{9n+9}{13} \right\rfloor$ for $0 \leq n \leq 99$, except for $n \in \{48,61,74,84,87,90,94,97\}$.
This arises because the rank of the chain space of $X_n^2$ in dimension $\left\lfloor \frac{9n+9}{13} \right\rfloor$ exceeds the sum of the ranks of the chain spaces in dimensions $\left\lfloor \frac{9n+9}{13} \right\rfloor - 1$ and $\left\lfloor \frac{9n+9}{13} \right\rfloor + 1$ for these values of $n$.
Further, $X_n^2$ is a wedge of spheres for $n \in \{0,1,2,3,4,5,7,8,11\}$.
\end{observation}  
As $n$ grows larger, the data suggests that the rank of this particular chain space ceases to ``typically'' exceed the sum of the ranks of the neighboring chain spaces.
This suggests that the behavior of $\Ind(\Delta_n^m)$ for ``small'' values of $n$, including many values of $n$ for which by-hand computations appear prohibitive, is not indicative of the general behavior of these complexes.

Thus, the topology of $\Ind(\Delta_n^m)$ remains generally mysterious.
It would be of interest to investigate the following two questions.
\begin{enumerate}
\item Does torsion occur in the homology of  $\Ind(\Delta_n^m)$?
If so, for which $p$ does $\mathbb{Z}/p\mathbb{Z}$ appear as a summand?
\item There is a natural action of the symmetric group $\mathfrak{S}_m$ on $\Ind(\Delta_n^m)$.
What is the $\mathfrak{S}_m$-module structure of $H_\ast(\Ind(\Delta_n^m);\mathbb{C})$?
\end{enumerate}

%


\bibliographystyle{plain}
\bibliography{Braun}

\end{document}